\documentclass[onecolumn,notitlepage,11pt]{article}
\usepackage{sectsty}
\usepackage{mathrsfs}
\usepackage{latexsym}
\usepackage{amsbsy}
\usepackage{amssymb}
\usepackage{verbatim}
\usepackage{amsthm}
\usepackage{amsfonts}
\usepackage{amsmath}
\usepackage{graphicx}
\usepackage{layout}
\usepackage{pb-diagram}
\usepackage{amscd}
\usepackage{anysize}
\usepackage{tikz}
\usepackage{graphicx}
\usepackage[all,cmtip]{xy}
\usepackage{youngtab}
\usepackage{hyperref}

\newcommand{\Q}{\mathbb{Q}}

\newcommand{\D}{\mathbb{D}}

\newcommand{\beq}{\begin{equation*}}
\newcommand{\eeq}{\end{equation*}}
\theoremstyle{definition}
\newtheorem{theorem}{Theorem}[section]
\newtheorem*{theorem*}{Theorem}
\newtheorem{lemma}[theorem]{Lemma}

\newtheorem{rmk}[theorem]{Remark}

\marginsize{3cm}{3cm}{2cm}{2cm} 
\title{\Large\bf On negatively curved bundles with hyperbolic fibers outside the Igusa stable range}
\author{Mauricio Bustamante \and Francis Thomas Farrell \and Yi Jiang}
\newcommand{\Addresses}{{
  \bigskip
  \footnotesize
\noindent\textsc{Mauricio Bustamante}\par\nopagebreak\textsc{Institut f\"ur Mathematik, Universit\"at Augsburg}\par\nopagebreak
 \texttt{EMAIL: Mauricio.BustamanteLondono@math.uni-augsburg.de}\\
\textsc{Francis Thomas Farrell}\par\nopagebreak\textsc{Yau Mathematical Sciences Center, Tsinghua University, Beijing, China}\par\nopagebreak
 \texttt{EMAIL: farrell@math.tsinghua.edu.cn}\\
\textsc{Yi Jiang}\par\nopagebreak\textsc{Yau Mathematical Sciences Center, Tsinghua University, Beijing, China}\par\nopagebreak
 \texttt{EMAIL: yjiang117@mail.tsinghua.edu.cn}
 }}
\date{}

\def\a{\alpha}
\def\b{\partial}
\def\MET{\mathcal{MET}}
\def\Q{\mathbb{Q}}
\def\mapD{\mbox{Map}(\mathbb{D}^n; \partial,\frac{\mbox{Top}(n)}{O(n)})}
\def\mapM{\mbox{Map}(M,\frac{\mbox{Top}(n)}{O(n)})}
\def\TOP{\mathbb{TOP}}
\def\DIFF{\mathbb{DIFF}}
\def\TDM{\frac{\mbox{Top}_0(M)}{\mbox{Diff}_0(M)}}
\def\TDD{\frac{\mbox{Top}(\mathbb{D}^n,\b)}{\mbox{Diff}(\mathbb{D}^n,\b)}}

\begin{document}
\maketitle

\begin{abstract}
We prove that the Teichm\"{u}ller space $\mathcal{T}^{<0}(M)$ of
negatively curved metrics on a hyperbolic manifold $M$ has
nontrivial $i$-th rational homotopy groups for some $i> \dim M$.
Moreover, some elements of infinite order in $\pi_i B\mbox{Diff}(M)$
can be represented by bundles over $S^i$ with fiberwise negatively
curved metrics.


\end{abstract}

\section{Introduction}

 Let $M$ be a closed smooth manifold. Denote by $\MET(M)$ the space of all smooth Riemannian metrics on $M$, endowed
 with the smooth topology and let $\mbox{Diff}_{0}(M)$ be the group of all smooth self-diffeomorphisms of $M$ which are homotopic to the identity $1_{M}$. The group $\mathcal{D}_{0}(M):=\mathbb{R}^+\times \mbox{Diff}_{0}(M)$ acts on $\MET(M)$ by scaling and pulling back metrics, i.e.
$$(\lambda,\phi)g=\lambda(\phi^{-1})^*g$$
for $g\in\MET(M)$ and $(\lambda,\phi)\in
\mathcal{D}_{0}(M)=\mathbb{R}^+\times \mbox{Diff}_{0}(M)$. The
quotient space $$\mathcal{T}(M):=\MET(M)/\mathcal{D}_{0}(M)$$ is
called the \textit{Teichm\"{u}ller space of all Riemannian metrics}
on $M$. Furthermore, if $M$ admits a Riemannian metric with negative
sectional curvature, the \textit{Teichm\"{u}ller space
$\mathcal{T}^{<0}(M)$ of all negatively curved metrics} on $M$ is
defined to be the quotient space $\MET^{<0}(M)/\mathcal{D}_{0}(M)$
with $\MET^{<0}(M)$ the subspace of $\MET(M)$ consisting of all
negatively curved metrics on $M$. Moreover,
in that case the action
of $\mathcal{D}_{0}(M)$ is free and $\mathcal{T}(M)$
is a model for the classifying space $B\mbox{Diff}_0(M)$ of
$\mbox{Diff}_0(M)$ (see \cite[Lemma 1.1]{FO09}).

It is proved in \cite{FO09} that if $M$ is real hyperbolic, the Teichm\"{u}ller space $\mathcal{T}^{<0}(M)$ is,
in general, not contractible. More precisely, they prove that the inclusion map $F: \mathcal{T}^{<0}(M)\rightarrow \mathcal{T}(M)$
(which forgets the negatively curved structure) is in general
homotopically nontrivial. Similar results are obtained in
\cite{S14,FS17} when $M$ is a Gromov--Thurston manifold or a
complex hyperbolic manifold. These results can be translated in
the language of bundle theory (c.f. \cite{FOGAFA,F}), which we now
recall. A smooth $M-$bundle $p:E\rightarrow B$ is said to be
\textit{negatively curved} if its concrete fibers
$p^{-1}(x), x\in B$ can be endowed with negatively curved
Riemannian metrics varying continuously from fiber to fiber.
The Teichm\"{u}ller space $\mathcal{T}(M)=B\mbox{Diff}_0(M)$
classifies equivalence classes of smooth bundles with abstract
fiber $M$ and equipped with a fiber homotopy trivialization,
whereas $\mathcal{T}^{<0}(M)$ classifies equivalence classes of
negatively curved bundles with abstract fiber $M$ and equipped
with a fiber homotopy trivialization. It follows from the results
obtained in \cite{FO09,S14,FS17} that there are negatively curved
bundles which are nontrivial as smooth bundles with abstract fiber
a hyperbolic manifold, a Gromov--Thurston manifold or a complex
hyperbolic manifold. However they all represent
elements of finite order
in the homotopy groups of $\mathcal{T}(M)$. 
We show in this paper that there are negatively curved bundles representing elements of infinite order in the homotopy groups of $\mathcal{T}(M)$. Consequently there are closed real hyperbolic manifolds $M$ such that
$\pi_{i}\mathcal{T}^{<0}(M)\otimes \mathbb{Q}\neq 0$ for some $i$ augmenting the main result of \cite{FO09}.

In order to ease the notation
we write $\Phi^{\Q}$ for $\Phi\otimes id_{\Q}$ when $\Phi$ is a homomorphism between abelian groups. Occasionally we denote $\pi_*(\cdot)\otimes \Q$ by $\pi_*^{\Q}(\cdot)$.

\begin{theorem}\label{app_Weiss class}There is a positive constant $c$ such that for every closed real hyperbolic manifold $M^{n}$ of dimension $n\geq c$ there exists a finite sheeted
cover $\tilde{M}$ of $M$ such that for every finite sheeted cover
$\hat{M}$ of $\tilde{M}$, the homomorphism
\beq
F_*^{\Q}:\pi_{i}\mathcal{T}^{<0}(\hat{M})\otimes \Q\rightarrow
\pi_{i}\mathcal{T}(\hat{M})\otimes \Q
\eeq
induced by the forget structure map is nontrivial for some $i\in[n+1,2n]$.
\end{theorem}

\begin{rmk}
  Note that the range $i\geq n+1$ in Theorem \ref{app_Weiss class} is outside the Igusa stable range which is $i\leq\min\{\frac{n-7}{2},\frac{n-4}{3}\}$(c.f. \cite{Igu88} and \cite[p.252]{Igu02}). We show in a companion paper \cite{BFJ} that inside Igusa stable range, negatively curved bundles over spheres can only represent elements of finite order in $\pi_i \mathcal{T}(M^n)$.
\end{rmk}

We prove Theorem \ref{app_Weiss class} in Section
\ref{se:app_mainth}. This theorem will follow from Theorem
\ref{mainth} below together with the discovery of nontrivial
rational Pontryagin classes in $H^{4m+4k}(B\mbox{Top}(2m);\Q)$ for
some $k>0$ by Weiss \cite{W15}.

 Let $\mbox{Diff}(\mathbb{D}^{n},\partial)$ be the group of all self-diffeomorphisms of $\mathbb{D}^n$ that restrict to the identity on the boundary $\partial \mathbb{D}^n$, equipped with the $C^{\infty}-$topology. Let $\Omega \mbox{Diff}(\mathbb{D}^{n},\partial)$ be the space of all continuous loops in $\mbox{Diff}(\mathbb{D}^{n},\partial)$ based at the identity $1_{\mathbb{D}^{n}}$. When necessary, we will assume the loops $t\mapsto f_{t}\in\mbox{Diff}(\mathbb{D}^{n},\partial)$
for $t\in [0,1]$ are smooth, i.e. the map
$(x,t)\mapsto f_{t}(x)$ is smooth. Note that the inclusion of the space of all smooth loops into $\Omega \mbox{Diff}(\mathbb{D}^{n},\partial)$ is a homotopy equivalence (c.f. \cite[p.55]{FO09}). Define a map
\beq
\alpha_{n}:\Omega \mbox{Diff}(\mathbb{D}^{n-1},\partial)\rightarrow \mbox{Diff}(\mathbb{D}^n,\b)
\eeq
 by $\a_n(f)(x,t)=(f_{t}(x),t)$ for $(x,t)\in \mathbb{D}^{n-1}\times [0,1]=\mathbb{D}^n$, where $f\in \Omega \mbox{Diff}(\mathbb{D}^{n-1},\b)$ is a loop based at $1_{\D^n}$.

\begin{theorem}\label{mainth}
  Let $i,n$ be two positive integers  with $i\geq n+1$. Assume that the composition of homomorphisms
\begin{equation*}
\xymatrix{\pi_{i+1}\mbox{Diff}(\mathbb{D}^{n-1},\partial)\otimes \Q~\ar[r]^-{\a_{n*}^{\Q}}&~\pi_{i}\mbox{Diff}(\mathbb{D}^{n},\partial)\otimes\Q~\ar[r]^-{\a_{n+1*}^{\Q}}&~\pi_{i-1}\mbox{Diff}(\mathbb{D}^{n+1},\partial)\otimes \Q
}
\end{equation*}
is nontrivial. Then for every closed real hyperbolic manifold $M$ of dimension $n$ there exists a finite sheeted
cover $\tilde{M}$ of $M$ such that for every finite sheeted cover
$\hat{M}$ of $\tilde{M}$, the homomorphism $F_*\otimes
id_{\Q}:\pi_{i+1}\mathcal{T}^{<0}(\hat{M})\otimes \Q\rightarrow
\pi_{i+1}\mathcal{T}(\hat{M})\otimes \Q$ induced by the forget
structure map is nontrivial.
\end{theorem}

\begin{rmk}
  The map $\a_{n}$ has appeared in Gromoll's fundamental work \cite{Gro66} on positive curvature problems and it has also been used by the second author and Ontaneda in \cite{FO09} to study the Teichm\"{u}ller space of negatively curved metrics.
\end{rmk}

  The proof of Theorem \ref{mainth} is presented in the next two sections.

\subsection*{Acknowledgements}
M. Bustamante is supported by the DFG Priority Programme
\textit{Geometry at Infinity} SPP 2026: \textit{Spaces and Moduli
Spaces of Riemannian Metrics with Curvature Bounds on compact and
non-compact Manifolds.}
 He acknowledges the hospitality of the Yau Mathematical Sciences Center at Tsinghua University in Beijing,
 where he was a postdoc when most of the results of this paper were obtained. Y. Jiang's research is partially supported by NSFC 11571343.

\section{Nontrivial elements in $\pi_{*}\mathcal{T}(M)\otimes \Q$}
The goal of this section is to construct nontrivial elements in $\pi^{\mathbb{Q}}_{i+1}\mathcal{T}(M)$ out of elements in $\pi^{\mathbb{Q}}_{i}\mbox{Diff}(\mathbb{D}^n,\b)$. Recall that a smooth manifold is \textit{parallelizable} if its tangent bundle is trivial.

Throughout the paper $\mbox{Top}(n)$ denotes the topological group
 with the compact open topology of all self-homeomorphisms of $\mathbb{R}^n$ and $O(n)\subset
\mbox{Top}(n)$ is the subgroup of all orthogonal transformations of
$\mathbb{R}^n$. Let $[X;A,Y;B]$ denote the set of all homotopy
classes of maps $(X,A)\rightarrow (Y,B)$ between pairs of
topological spaces.

\begin{lemma}\label{paralnonzero}
  Let $M^n$ be a closed parallelizable manifold which admits a Riemannian metric of nonpositive sectional curvature and let $\iota: \mbox{Diff}(\mathbb{D}^n,\b)\rightarrow \mbox{Diff}_0(M)$ be the map given by extending each $f\in \mbox{Diff}(\mathbb{D}^n,\b)$ by the identity outside a fixed embedding
$\nu:\mathbb{D}^n\subset M$. If $i\geq n\geq 5$, then the homomorphism $\iota_*: \pi_i \mbox{Diff}(\mathbb{D}^n,\b)\rightarrow \pi_i \mbox{Diff}_0(M)$ is injective.
\end{lemma}

\begin{proof}
Let $\mapD$ denote the space of all continuous maps
$(\mathbb{D}^n,\partial\D^n)\rightarrow\left(\frac{\mbox{Top}(n)}{O(n)},*\right)$
with the base point $*\in \frac{\mbox{Top}(n)}{O(n)}$ the coset of the identity element. Define a map $\eta:\mapD\rightarrow\mapM$ by extending each $f\in\mapD$ via the constant map outside $\mathbb{D}^n\subset M$.

  \textbf{Claim:} For every $i\geq 1$, if the homomorphism $\eta_*:\pi_{i+1} \mapD\rightarrow\pi_{i+1} \mapM$ is injective, then $\iota_*:
\pi_i \mbox{Diff}(\mathbb{D}^n,\b)\rightarrow \pi_i \mbox{Diff}_0(M)$ is injective.

\begin{proof}[Proof of Claim]\let\qed\relax
  Firstly, let $\TOP_0(M)$ (resp. $\DIFF_0(M)$) denote the
singular complex (resp. singular differentiable complex) of
$\mbox{Top}_0(M)$ (resp. $\mbox{Diff}_0(M)$) and let
\beq
\frac{\mbox{Top}_0(M)}{\mbox{Diff}_0(M)}:=\left|\frac{\TOP_0(M)}{\DIFF_0(M)}\right|
\eeq
 be the geometric realization of the simplicial set $\frac{\TOP_0(M)}{\DIFF_0(M)}$. The space $\TDD$ is similarly defined. Note that the map $\iota: \mbox{Diff}(\mathbb{D}^n,\b)\rightarrow \mbox{Diff}_0(M)$ extends canonically to a map $\mbox{Top}(\mathbb{D}^n,\b)\rightarrow \mbox{Top}_0(M)$, which
in turn gives rise to a map
$\overline{\iota}:\TDD\rightarrow \TDM$. Then consider the commutative diagram:
  \begin{equation}\label{diag:inj}
  \xymatrix{\pi_{i+1}\TDD\ar[r]^{\b_{\mathbb{D}^n}}_{\cong}\ar[d]_{\overline{\iota}_*}&\pi_i|\DIFF(\mathbb{D}^n,\b)|\ar[r]_{\cong}\ar[d]_{S\iota_*}&\pi_i\mbox{Diff}(\mathbb{D}^n,\b)\ar[d]_{\iota_*}\\
               \pi_{i+1}\TDM\ar[r]^{\b_M}&\pi_i|\DIFF_0(M)|\ar[r]_{\cong}&\pi_i \mbox{Diff}_0(M)}
  \end{equation}
   where $S\iota$ is induced by $\iota$ in the canonical way and $\b_{\mathbb{D}^n}$ (resp. $\b_M$) denotes the connecting homomorphism in the homotopy exact sequence associated to the fibration $|\TOP(\mathbb{D}^n,\b)|\rightarrow \TDD$ (resp. $|\TOP_0(M)|\rightarrow \TDM$). Since $M$ is closed and admits a Riemannian metric with nonpositive sectional curvature and $n\geq 5$, then by \cite[Corollary 4.3]{FW}, the canonical map $\mbox{Diff}(M)\rightarrow \mbox{Top}(M)$ induces surjective homomorphisms on $\pi_i$ for all $i\geq 2$. Furthermore, since the canonical maps
   \begin{gather*}
     \xymatrix{|\TOP_0(M)|\ar[r]_{\cong}&\mbox{Top}_0(M)\ar[r]&\mbox{Top}(M)} \\
     \xymatrix{|\DIFF_0(M)|\ar[r]_{\cong}&\mbox{Diff}_0(M)\ar[r]&\mbox{Diff}(M)}
   \end{gather*}
induce isomorphisms on $\pi_i$ for $i\geq 1$,
then the map $|\DIFF_0(M)|\rightarrow |\TOP_0(M)|$
also induces epimorphisms on $\pi_i$ for $i\geq 2$,
and hence $\b_M:\pi_{i+1}\TDM\rightarrow \pi_i |\DIFF_0(M)|$
is injective, for $i\geq 1$, by the homotopy exact
sequence associated to the fibration
$|\TOP_0(M)|\rightarrow \TDM$.
Consequently, by the commutative diagram (\ref{diag:inj}),
to prove $\iota_*$ is injective, it suffices to show that
$\overline{\iota}_*: \pi_{i+1}\TDD\rightarrow \pi_{i+1}\TDM$
is injective. Hence the claim follows from the
commutativity of the next diagram:
\beq
\xymatrix{\pi_{i+1}\TDD\ar[r]^-{\mu(\mathbb{D}^n)_*}_-{\cong}\ar[d]_{\overline{\iota}_*}&\pi_{i+1}\mapD\ar[d]^{\eta_*}\\
   \pi_{i+1}\TDM\ar[r]^-{\mu(M)_*}_-{\cong}&\pi_{i+1}\mapM}
\eeq
   where $\mu(\mathbb{D}^n)_*$, $\mu(M)_*$ the Morlet isomorphisms\cite[Theorem 4.2]{BL}.
   \end{proof}

 Thus it remains to show $\eta_*$ is injective. Let $\hat{\eta}:M\times \mathbb{D}^{i+1}/\b \rightarrow S^{n+i+1}=\mathbb{D}^n\times \mathbb{D}^{i+1}/\b$ be the degree 1 map which collapses the union of $M\times \b \mathbb{D}^{i+1}$ and the complement of $\nu\times id_{\mathbb{D}^{i+1}}: \mathbb{D}^n\times \mathbb{D}^{i+1}\rightarrow M\times \mathbb{D}^{i+1}$ to a point. Then $\eta_*$ can be identified with the map $\hat{\eta}^*:[S^{n+i+1};*,\frac{\mbox{Top}(n)}{O(n)};*]\rightarrow [M\times \mathbb{D}^{i+1};\b,\frac{\mbox{Top}(n)}{O(n)};*]$. Since $i\geq n$, then by Whitney embedding theorem, there is an embedding $M\subset S^{n+i+1}$ with a closed tubular neighborhood $N(M)$, which induces a degree 1 map $q:S^{n+i+1}\rightarrow N(M)/\b$ collapsing the complement of $N(M)\subset S^{n+i+1}$ into a point. Since $M$ is parallelizable, then $N(M)=M\times \mathbb{D}^{i+1}$.
    ~Now the composition $\hat{\eta}\circ q: S^{n+i+1}\rightarrow M\times \mathbb{D}^{i+1}/\b\rightarrow S^{n+i+1}$ is of degree 1 and hence homotopic to the identity. This shows that $q^*\circ \hat{\eta}^*$ is the identity on $[S^{n+i+1};*,\mbox{Top}(n)/O(n);*]$. Consequently, $\hat{\eta}^*$ and hence $\eta_*$ are injective. This completes the proof of the lemma.
\end{proof}
Recall that a smooth manifold is \textit{stably parallelizable} if the Whitney sum of its tangent bundle with a trivial line bundle is trivial.
\begin{lemma}\label{nonzero}
  Let $M^n$ be a closed stably parallelizable manifold admitting a Riemannian metric with nonpositive sectional curvature and assume $i\geq n+1$. If the image of a class $x\in\pi_{i}^\mathbb{Q}\mbox{Diff}(\mathbb{D}^{n},\partial)$ under the Gromoll-map $$\xymatrix{\pi_{i}^\mathbb{Q}\mbox{Diff}(\mathbb{D}^{n},\partial)\ar[r]^-{\a_{n+1*}^{\Q}}&\pi_{i-1}^\mathbb{Q}\mbox{Diff}(\mathbb{D}^{n+1},\partial)}$$
is nonzero, then $\iota_{*}^{\Q}x\neq
0\in\pi_{i}^\mathbb{Q}\mbox{Diff}_0(M)$ where $\iota:
\mbox{Diff}(\mathbb{D}^n,\b)\rightarrow \mbox{Diff}_0(M)$ is the map
given by extending each $f\in \mbox{Diff}(\mathbb{D}^n,\b)$ by the
identity outside a fixed embedding $\nu:
\mathbb{D}^n\hookrightarrow M$.
\end{lemma}
\begin{proof}
  The theorem is vacuous if $n=1,n=2$ by Smale\cite{S59} and if $n=3$ by Hatcher\cite{H}. Hence we may assume $n\geq 4$. Analogous to the map $\a_{n+1}$, a map $\sigma: \Omega \mbox{Diff}(M)\rightarrow \mbox{Diff}(M\times [0,1],\b)$ can be defined such that the following diagram commutes:
\begin{equation}\label{diag:nonzero}
\xymatrix{\pi_{i}\mbox{Diff}(\mathbb{D}^n,\b)\ar[r]^-{\a_{n+1*}}\ar[d]_{\iota_*}&\pi_{i-1}\mbox{Diff}(\mathbb{D}^{n+1},\b)\ar[d]^{\tau_*}\\
\pi_{i}\mbox{Diff}_0(M)\ar[r]^-{\sigma_*}&\pi_{i-1}\mbox{Diff}_0(M\times [0,1],\b)}
\end{equation}
where $\tau$ is the map given by extending each diffeomorphism in
$\mbox{Diff}(\mathbb{D}^{n+1},\b)$ via the identity outside the
embedding $$\xymatrix{\mathbb{D}^{n+1}=\mathbb{D}^n\times
[0,1]~\ar[r]^-{\nu\times id_{[0,1]} }&~M\times [0,1]}.$$ By the
commutativity of the diagram (\ref{diag:nonzero}), it suffices to
show $\tau_*^{\Q}\a_{n+1*}^{\Q}x\neq 0$. This can be seen as
follows: Let $q:\mbox{Diff}_0(M\times [0,1],\b)\rightarrow
\mbox{Diff}_0(M\times S^1)$ be the map given by extending each
diffeomorphism in $\mbox{Diff}_0(M\times [0,1],\b)$ via the identity
outside a fixed embedding $M\times [0,1]\subset M\times S^1$. Since
$M\times S^1$ is parallelizable and admits a Riemannian metric with
nonpositive sectional curvature, then Lemma \ref{paralnonzero}
yields that $(q\circ
\tau)_*:\pi_{i-1}\mbox{Diff}(\mathbb{D}^{n+1},\b)\rightarrow
\pi_{i-1}\mbox{Diff}_0(M\times S^1)$ is injective, which implies
that $\tau_*$ is injective and so must be $\tau_*\otimes id_{\Q}$.
Hence $$\tau_*^{\Q}\a_{n+1*}^{\Q}x\neq 0.$$ This completes the
proof.
\end{proof}

\begin{rmk}
  The condition $\a_{n+1*}^{\Q}x\neq 0$ is never satisfied inside Igusa stable range by \cite{FH78}.
\end{rmk}

\section{Elements in the image of the forgetful homomorphism  $F^{\Q}: \pi_{*}\mathcal{T}^{<0}(M)\otimes \Q\rightarrow \pi_{*}\mathcal{T}(M)\otimes \Q$}
In this section, we give a sufficient condition (in Lemma \ref{FO} below)
for an element in $\pi^{\mathbb{Q}}_{*}\mathcal{T}(M)$ to belong to
the image of $F_*^{\Q}:
\pi^{\mathbb{Q}}_{*}\mathcal{T}^{<0}(M)\rightarrow
\pi^{\mathbb{Q}}_{*}\mathcal{T}(M)$ and use it together with Lemma
\ref{nonzero} to prove Theorem \ref{mainth}.

For any embedding
$\nu:\mathbb{D}^n\rightarrow M$, let
$$\iota(\nu):\mbox{Diff}(\mathbb{D}^n,\b)\rightarrow
\mbox{Diff}_0(M)$$ be the map given by extending each $\phi\in
\mbox{Diff}(\mathbb{D}^n,\b)$ by the identity outside the embedding
$\nu:\mathbb{D}^n\rightarrow M$. The following lemma is a direct
consequence of \cite[Theorem 2]{FO09}.
\begin{lemma}\label{FO}
Let $i$ be a nonnegative integer and $n$ be a positive integer.
Suppose that a class
$\overline{x}\in\pi_{i}\mbox{Diff}(\mathbb{D}^n,\b)$ lies in the
image of the Gromoll homomorphism
$$\alpha_{n*}:\pi_{i+1}\mbox{Diff}(\mathbb{D}^{n-1},\b)\rightarrow\pi_{i}\mbox{Diff}(\mathbb{D}^n,\b).$$
Then there exists a real number $r>0$ with the following property:
for any closed real hyperbolic manifold $(M,g)$ with injectivity
radius at least $3r$ at some point, there is an embedding
$\nu:\mathbb{D}^n\rightarrow M$ such that the class
$$\iota(\nu)_*\overline{x}\in
\pi_i\mbox{Diff}_0(M)=\pi_{i+1}\mathcal{T}(M)$$ is in the image of
the forget structure homomorphism $F_*:
\pi_{i+1}\mathcal{T}^{<0}(M)\rightarrow \pi_{i+1}\mathcal{T}(M)$.
\end{lemma}

  The proof of this lemma is given below, after we make the following considerations. Let $$\Gamma_{g}: \mbox{Diff}_0(M)\rightarrow \MET^{<0}(M)$$ denote the orbit map given by $\Gamma_{g}(\phi)=(\phi^{-1})^*g$. Since the sequence $$\xymatrix{\pi_{i+1}\mathcal{T}^{<0}(M)\ar[r]^-{F_*}&\pi_{i+1}\mathcal{T}(M)=\pi_{i}\mbox{Diff}_0(M)\ar[r]^-{\Gamma_{g*}}&\pi_{i}\MET^{<0}(M)}$$ is exact (c.f. \cite{FO09,F}), then to show Lemma \ref{FO} it suffices to prove that there is an
embedding $\nu:\mathbb{D}^n\rightarrow M$ such that the image of
$\overline{x}$ under the composition
$$\xymatrix{\pi_{i}\mbox{Diff}(\mathbb{D}^n,\b)\ar[r]^-{\iota(\nu)_*}&\pi_{i}\mbox{Diff}_0(M)\ar[r]^-{\Gamma_{g*}}&\pi_{i}\MET^{<0}(M)}$$
is zero. To see how this follows, we review some notions from
\cite{FO09}.

Let $\mbox{Diff}_0(S^{n-1}\times [1,2],\b)$ denote the group of all
self-diffeomorphisms of $S^{n-1}\times [1,2]$ that are the identity
near $S^{n-1}\times \{1,2\}$ and are homotopic to the identity by a
homotopy that is constant near $S^{n-1}\times\{1,2\}$. Let
$\mathcal{G}$ be the subgroup of $\mbox{Diff}_0(S^{n-1}\times
[1,2],\b)$ whose elements are all smooth isotopies $\phi$ of
$S^{n-1}$, namely $\phi(\cdot, t)\in \mbox{Diff}(S^{n-1})$ for all
$t\in[1,2]$. Let $N$ be a real hyperbolic manifold of dimension $n$,
with a geodesic ball $B$ of radius $2r$ centered at some point $p\in
N$, hence $B\backslash p$ can be identified with $S^{n-1}\times
(0,2r]$. Furthermore, identify $S^{n-1}\times [r,2r]$ with
$S^{n-1}\times [1,2]$ by $(x,t)\mapsto (x,\frac{t}{r})$ for any
$(x,t)\in[r,2r]$. Then under these identifications, a map
$$\Lambda(N,p,r):\mbox{Diff}_0(S^{n-1}\times [1,2],\b)\rightarrow
\mbox{Diff}_0(N)$$ can be defined by extending each $\varphi\in
\mbox{Diff}_0(S^{n-1}\times [1,2],\b)$ via the identity outside
$S^{n-1}\times [r,2r]\subset N$. Now we can restate a special case
of \cite[Theorem 2]{FO09} as follows:

\begin{theorem}[Farrell-Ontaneda]\label{th:FO09}Given a compact subset $K\subset \mathcal{G}$, there is a real number $r>0$ such that the following holds: let $(N,g)$ be a closed real hyperbolic manifold and let $p\in N$ with radius of injectivity at $p$ larger than $3r$. Then the map $$\xymatrix{K~\ar[r]^-{\Lambda(N,p,r)}&~\mbox{Diff}_0(N)~\ar[r]^-{\Gamma_g}&~\MET^{<0}(N)}$$
is homotopic to a constant map.
\end{theorem}

\begin{proof}[Proof of Lemma \ref{FO}] Identify $\mathbb{D}^{n-1}$ with the northern hemisphere of the sphere $S^{n-1}$. Then this induces an inclusion $\mathbb{D}^n=\mathbb{D}^{n-1}\times[1,2]\rightarrow S^{n-1}\times[1,2]$, which gives rise to a map $$j:\mbox{Diff}(\mathbb{D}^n,\b)\rightarrow \mbox{Diff}_0(S^{n-1}\times[1,2],\b).$$
Let $f:S^i\rightarrow \Omega \mbox{Diff}(\mathbb{D}^{n-1},\b)$ be a
continuous map such that $\a_n\circ f$ represents the homotopy class
$\overline{x}$ and let $K=\{(j\circ \a_n\circ f)(u)|u\in S^i\}$.
Then $K$ is a compact subset of $\mathcal{G}$ and hence Theorem
\ref{th:FO09} applies. Let $r$ be the real number in Theorem
\ref{th:FO09}. By the assumption of Lemma \ref{FO} we can find a
$p\in M$ with radius of injectivity at $p$ larger than $3r$. Theorem
\ref{th:FO09} implies that the homomorphism $(\Gamma_g\circ
\Lambda(M,p,r))_*: \pi_i K\rightarrow \pi_i\MET^{<0}(M)$ is trivial.
In particular, if we define $\psi:S^i\rightarrow K$ by
$\psi(u)=j(\a_n(f(u)))$ then the composition $\Gamma_g\circ
\Lambda(M,p,r)\circ\psi$ is null-homotopic. Let
$\iota:=\Lambda(M,p,r)\circ j$, then $\iota$ is obviously equal to
$\iota(\nu)$ for some embedding $\nu:\mathbb{D}^n\rightarrow M$.
Therefore the composition of maps $\Gamma_g\circ\iota\circ\a_n\circ
f$ is null-homotopic, which implies
$\Gamma_{g*}\iota_*\overline{x}=0$. This completes the proof of the
lemma.

\end{proof}

Now we can prove Theorem \ref{mainth}.

\begin{proof}[Proof of Theorem \ref{mainth}]

By the assumption of Theorem \ref{mainth}, there exists
$y\in\pi_{i+1}^{\Q}\mbox{Diff}(\mathbb{D}^{n-1},\b)$ such
that $(\a_{n+1}\circ \a_n)_*^{\Q}y\neq 0$ and
hence there is $N>>0$ and $\overline{y}\in
\pi_{i+1}\mbox{Diff}(\mathbb{D}^{n-1},\b)$ such that
$$N\cdot y=\overline{y}\otimes 1\in \pi_{i+1}^{\Q}\mbox{Diff}(\mathbb{D}^{n-1},\b).$$
Let $\overline{x}=\a_{n*}\overline{y}\in \pi_{i}\mbox{Diff}(\mathbb{D}^n,\b)$,
then Lemma \ref{FO} applies.
Given a closed real hyperbolic manifold $N$, there is a finite sheeted cover of $M$ which is
stably parallelizable (the proof of this result is sketched in
\cite[p.553]{Su79}. Another proof appears in \cite{Okun}).
 We can further pass to another
finite sheeted cover $\tilde{M}$ of the latter so that
$\tilde{M}$ has a geodesic ball with large enough
radius as required in Lemma \ref{FO}
(c.f. \cite[p.901]{FJ89}). Then for any finite sheeted
cover $\hat{M}$ of $\tilde{M}$ there is an embedding
$\nu:\mathbb{D}^n\rightarrow \hat{M}$ such that $\iota(\nu)_*\overline{x}$ lives in the image of the homomorphism $F_*:\pi_{i+1}\mathcal{T}^{<0}(\hat{M})\rightarrow\pi_{i+1}\mathcal{T}(\hat{M})$.
Let now $x=\overline{x}\otimes 1\in\pi_i^{\Q}\mbox{Diff}(\mathbb{D}^n,\b)$. Then $\iota(\nu)_*^{\Q}x=\iota(\nu)_*\overline{x}\otimes 1$ lies in the image of the homomorphism $F_*^{\Q}:\pi_{i+1}^{\Q}\mathcal{T}^{<0}(\hat{M})\rightarrow\pi_{i+1}^{\Q}\mathcal{T}(\hat{M})$ and $$\a_{n+1*}^{\Q}x=N(\a_{n+1}\circ\a_n)_*^{\Q}y\neq 0.$$ Also by Lemma \ref{nonzero}, the class $\iota_*^{\Q}x\in \pi_i^{\Q}\mbox{Diff}_0(\hat{M})$ is nonzero. This completes the proof of the theorem.
\end{proof}

\section{An example coming from Pontryagin classes in $H^*(B\mbox{Top}(m);\mathbb{Q})$}\label{se:app_mainth}

In this section we apply a recent striking result of M. Weiss
\cite{W15} to exhibit examples where the conditions of Theorem
\ref{mainth} are satisfied. As a byproduct we will obtain a proof of
Theorem \ref{app_Weiss class}.

Let $B\mbox{Top}(m)$ denote the classifying space of
$\mbox{Top}(m)$. The inclusion
$\mbox{Top}(m)\hookrightarrow\mbox{Top}(m+1)$ induces the map
$B\mbox{Top}(m)\hookrightarrow B\mbox{Top}(m+1)$. Let
$B\mbox{Top}:=\varinjlim B\mbox{Top}(m)$ and
$j_l:B\mbox{Top}(l)\rightarrow B\mbox{Top}$ be the canonical map.
Denote by $p_i\in H^{4i}(B\mbox{Top};\Q)$ the $i-$th universal
rational Pontryagin class.
\begin{theorem}[Weiss \label{th:Weiss}\cite{W15}]
  There exist positive constants $c_1$ and $c_2$ such that, for all positive integers $m$ and $k$ where $m\geq c_1$ and $k<5m/4-c_2$, the rational Pontryagin class $p_{m+k}\in H^{4m+4k}(B\mbox{Top}(2m);\Q)$ evaluates nontrivially on $\pi^{\Q}_{4m+4k}(B\mbox{Top}(2m))$.
\end{theorem}

Let $c_1$ and $c_2$ be the positive numbers given by Theorem \ref{th:Weiss}, then we have the following lemma.
\begin{lemma}\label{le:app}
  Let $n$ and $i$ be nonnegative integers satisfying the conditions $n\geq 2c_1+2$, $i\equiv-n+2\mod 4$ and $n\leq i<\frac{7}{2}n-11-4c_2$, then the composite \begin{equation}\label{diag:W}
    \xymatrix{\pi^{\mathbb{Q}}_{i+1}\mbox{Diff}(\mathbb{D}^{n-1},\partial)\ar[r]^-{\a_{n*}^{\Q}}&\pi^{\mathbb{Q}}_{i}\mbox{Diff}(\mathbb{D}^{n},\partial)\ar[r]^-{\a_{n+1*}^{\Q}}&\pi^{\mathbb{Q}}_{i-1}\mbox{Diff}(\mathbb{D}^{n+1},\partial)}
  \end{equation} is nontrivial.
\end{lemma}
\begin{proof}
  Let $\mbox{Top}(l)/O(l)$ be the homotopy fiber of the forget structure map $BO(l) \rightarrow B\mbox{Top}(l)$. Consider the commutative diagram
  $$\xymatrix{\pi_{i+1}^{\Q}\mbox{Diff}(\mathbb{D}^{n-1},\b)\ar[r]^-{\a_{n*}^{\Q}}\ar[d]_{\mu_{n-1*}^{\Q}}^{\cong}&\pi_i^{\Q}\mbox{Diff}(\mathbb{D}^n,\b)\ar[r]^-{\a_{n+1*}^{\Q}}\ar[d]_{\mu_{n*}^{\Q}}^{\cong}&\pi_{i-1}^{\Q}\mbox{Diff}(\mathbb{D}^{n+1},\b)\ar[d]_{\mu_{n+1*}^{\Q}}^{\cong}\\
  \pi_{i+n+1}^{\Q}\mbox{Top}(n-1)/O(n-1)\ar[r]&\pi_{i+n+1}^{\Q}\mbox{Top}(n)/O(n)\ar[r]&\pi_{i+n+1}^{\Q}\mbox{Top}(n+1)/O(n+1)\\
  \pi_{i+n+2}^{\Q}B\mbox{Top}(n-1)\ar[r]^{i_{n*}^{\Q}}\ar[u]^{\b_*^{\Q}}_{\cong}&\pi_{i+n+2}^{\Q}B\mbox{Top}(n)\ar[r]^{i_{n+1*}^{\Q}}\ar[u]^{\b_*^{\Q}}_{\cong}&\pi_{i+n+2}^{\Q}B\mbox{Top}(n+1)\ar[u]^{\b_*^{\Q}}_{\cong}}$$
  where the vertical maps $\mu_{l*}^{\Q}, n-1\leq l\leq n+1$ are the Morlet isomorphisms\cite[Theorem 4.2]{BL}; the horizontal maps in the middle and bottom rows are induced by the standard stabilizations $\mbox{Top}(n-1)\rightarrow \mbox{Top}(n)$ and $\mbox{Top}(n)\rightarrow \mbox{Top}(n+1)$, and the vertical maps $\b^{\Q}_*$ are the connecting homomorphisms in the homotopy exact sequences associated to the corresponding fibrations; the commutativity of the two top squares is guaranteed by \cite[Theorem 1.3]{B73}. Since $i\geq n$, then the rational homotopy groups $\pi^{\Q}_{i+n+2}BO(l)$ and $\pi^{\Q}_{i+n+1}BO(l)$ are zero for $n-1\leq l\leq n+1$ and hence the connecting homomorphisms $\b^{\Q}_*$ are isomorphisms. Hence to prove that $\a_{n+1*}^{\Q}\circ\a_{n*}^{\Q}$ is nontrivial, it suffices to show $i_{n+1*}^{\Q}\circ i_{n*}^{\Q}$ is nontrivial.

  We claim that for any pair of integers $n,i$ such that $n\geq 2c_1+2$, $i\equiv-n+2\mod 4$ and $n\leq i<\frac{7}{2}n-11-4c_2$, there is an integer $m$ such that $n-1\geq 2m$ and such that the rational Pontryagin class $j_{2m}^*p_{\frac{i+n+2}{4}}\in H^{i+n+2}(B\mbox{Top}(2m);\Q)$ evaluates nontrivially on $\pi_{i+n+2}B\mbox{Top}(2m)$. In fact,
  considers a pair of integers
  $$(m,k)=\left\{
     \begin{array}{ll}
       (\frac{n+1}{2},\frac{i-n}{4}), & \hbox{if $n$ is odd.} \\
       (\frac{n-2}{2},\frac{i-n+6}{4} ), & \hbox{if $n$ is even.}
     \end{array}
   \right.
  $$
Then $j_{2m}^*p_{\frac{i+n+2}{4}}=j_{2m}^*p_{m+k}\in H^{i+n+2}(B\mbox{Top}(2m);\Q)=H^{4m+4k}(B\mbox{Top}(2m);\Q)$ lives in the family of Pontryagin classes in Theorem \ref{th:Weiss} and hence the claim follows.

Let now $i:B\mbox{Top}(2m)\rightarrow B\mbox{Top}(n-1)$ be induced
by the standard stabilization $\mbox{Top}(2m)\rightarrow
\mbox{Top}(n-1)$. Since $$j_{n+1}\circ i_{n+1}\circ i_n\circ
i=j_{2m}:\mbox{Top}(2m)\rightarrow \mbox{Top}.$$ We have by the
claim above that there is $x\in \pi_{i+n+2}^{\Q}B\mbox{Top}(2m)$
such that
\beq
\left<j_{n+1}^*p_{\frac{i+n+2}{4}},[i_{n+1*}^{\Q}i_{n*}^{\Q}i_*^{\Q}x]\right>=
\left<i^*i_n^*i_{n+1}^*j_{n+1}^*p_{\frac{i+n+2}{4}},[x]\right>=
\left<j_{2m}^*p_{\frac{i+n+2}{4}},[x]\right>\neq 0
\eeq
where the bracket denotes the Kronecker product and $[x]\in
H_{i+n+2}(B\mbox{Top}(2m);\Q)$ (resp.\\
$[i_{n+1*}i_{n*}^{\Q}i_*^{\Q}x]\in H_{i+n+2}(B\mbox{Top}(n+1);\Q)$)
denotes the image of $x$ (resp. $i_{n+1*}^{\Q}i_{n*}^{\Q}i_*^{\Q}x$)
under the Hurewicz homomorphisms. This implies $i_{n+1*}^{\Q}\circ
i_{n*}^{\Q}$ is nontrivial and completes the proof.
\end{proof}

\begin{proof}[Proof of Theorem \ref{app_Weiss class}]
  Let $n$ and $i$ be nonnegative integers satisfy that $n\geq 2c_1+2$, $i\equiv-n+2\mod 4$ and $n+1\leq i<\frac{7}{2}n-11-4c_2$, Then by Lemma \ref{le:app} and Theorem \ref{mainth} we have that for any large enough finite sheeted cover $M$ of a closed real hyperbolic manifold, the forgetful map $\pi^{\mathbb{Q}}_{i+1}\mathcal{T}^{<0}(M)\rightarrow \pi^{\mathbb{Q}}_{i+1}\mathcal{T}(M)$ is nonzero.
The result follows if we take $c=\max\{2c_1+2,\frac{25+8c_2}{3}\}$.
\end{proof}

{\footnotesize \ }

{\footnotesize \
\bibliographystyle{alpha}
\bibliography{bib}

\begin{thebibliography}{Sma59}

\bibitem[BFJ17]{BFJ}
M.~Bustamante, F.~T. Farrell, and Y.~Jiang.
\newblock Negatively curved bundles in the {I}gusa stable range.
\newblock {\em arXiv:1711.11315}, 2017.

\bibitem[BL74]{BL}
Dan Burghelea and Richard Lashof.
\newblock The homotopy type of the space of diffeomorphisms. {I}, {II}.
\newblock {\em Trans. Amer. Math. Soc.}, 196:1--36; ibid. 196\ (1974), 37--50,
  1974.

\bibitem[Bur73]{B73}
Dan Burghelea.
\newblock On the homotopy type of {${\rm Diff}(M^{n})$} and connected problems.
\newblock {\em Ann. Inst. Fourier (Grenoble)}, 23(2):3--17, 1973.
\newblock Colloque International sur l'Analyse et la Topologie Diff\'erentielle
  (Colloq. Internat. CNRS, No. 210, Strasbourg, 1972).

\bibitem[Far16]{F}
F.~T. Farrell.
\newblock Bundles with extra geometric or dynamic structure.
\newblock In {\em The legacy of {B}ernhard {R}iemann after one hundred and
  fifty years. {V}ol. {I}}, volume~35 of {\em Adv. Lect. Math. (ALM)}, pages
  223--250. Int. Press, Somerville, MA, 2016.

\bibitem[FH78]{FH78}
F.~T. Farrell and W.~C. Hsiang.
\newblock On the rational homotopy groups of the diffeomorphism groups of
  discs, spheres and aspherical manifolds.
\newblock In {\em Algebraic and geometric topology ({P}roc. {S}ympos. {P}ure
  {M}ath., {S}tanford {U}niv., {S}tanford, {C}alif., 1976), {P}art 1}, Proc.
  Sympos. Pure Math., XXXII, pages 325--337. Amer. Math. Soc., Providence,
  R.I., 1978.

\bibitem[FJ89]{FJ89}
F.~T. Farrell and L.~E. Jones.
\newblock Negatively curved manifolds with exotic smooth structures.
\newblock {\em J. Amer. Math. Soc.}, 2(4):899--908, 1989.

\bibitem[FO09]{FO09}
F.~Thomas Farrell and Pedro Ontaneda.
\newblock The {T}eichm\"uller space of pinched negatively curved metrics on a
  hyperbolic manifold is not contractible.
\newblock {\em Ann. of Math. (2)}, 170(1):45--65, 2009.

\bibitem[FO10]{FOGAFA}
Tom Farrell and Pedro Ontaneda.
\newblock Teichm\"uller spaces and negatively curved fiber bundles.
\newblock {\em Geom. Funct. Anal.}, 20(6):1397--1430, 2010.

\bibitem[FS17]{FS17}
F.~Thomas Farrell and Gangotryi Sorcar.
\newblock Teichm\"uller space of negatively curved metrics on complex
  hyperbolic manifolds is not contractible.
\newblock {\em Sci. China Math.}, 60(4):569--580, 2017.

\bibitem[FW91]{FW}
Steven~C. Ferry and Shmuel Weinberger.
\newblock Curvature, tangentiality, and controlled topology.
\newblock {\em Invent. Math.}, 105(2):401--414, 1991.

\bibitem[Gro66]{Gro66}
Detlef Gromoll.
\newblock Differenzierbare {S}trukturen und {M}etriken positiver {K}r\"ummung
  auf {S}ph\"aren.
\newblock {\em Math. Ann.}, 164:353--371, 1966.

\bibitem[Hat83]{H}
Allen~E. Hatcher.
\newblock A proof of the {S}male conjecture, {${\rm Diff}(S^{3})\simeq {\rm
  O}(4)$}.
\newblock {\em Ann. of Math. (2)}, 117(3):553--607, 1983.

\bibitem[Igu88]{Igu88}
Kiyoshi Igusa.
\newblock The stability theorem for smooth pseudoisotopies.
\newblock {\em $K$-Theory}, 2(1-2):vi+355, 1988.

\bibitem[Igu02]{Igu02}
Kiyoshi Igusa.
\newblock {\em Higher {F}ranz-{R}eidemeister torsion}, volume~31 of {\em AMS/IP
  Studies in Advanced Mathematics}.
\newblock American Mathematical Society, Providence, RI; International Press,
  Somerville, MA, 2002.

\bibitem[Oku01]{Okun}
Boris Okun.
\newblock Nonzero degree tangential maps between dual symmetric spaces.
\newblock {\em Algebr. Geom. Topol.}, 1:709--718, 2001.

\bibitem[Sma59]{S59}
Stephen Smale.
\newblock Diffeomorphisms of the {$2$}-sphere.
\newblock {\em Proc. Amer. Math. Soc.}, 10:621--626, 1959.

\bibitem[Sor14]{S14}
Gangotryi Sorcar.
\newblock Teichm\"uller space of negatively curved metrics on
  {G}romov-{T}hurston manifolds is not contractible.
\newblock {\em J. Topol. Anal.}, 6(4):541--555, 2014.

\bibitem[Sul79]{Su79}
Dennis Sullivan.
\newblock Hyperbolic geometry and homeomorphisms.
\newblock In {\em Geometric topology ({P}roc. {G}eorgia {T}opology {C}onf.,
  {A}thens, {G}a., 1977)}, pages 543--555. Academic Press, New York-London,
  1979.

\bibitem[Wei16]{W15}
Michael~S. Weiss.
\newblock Dalian notes on pontryagin classes.
\newblock {\em Available at arXiv:1507.00153v3}, 2016.

\end{thebibliography}
}
\Addresses
\end{document}